\documentclass[11pt]{amsart}

\usepackage[english]{babel}
\usepackage[utf8]{inputenc}
\usepackage[T1]{fontenc}

\usepackage{amsmath}
\usepackage{amssymb}
\usepackage{amsfonts}
\usepackage{amssymb}
\usepackage{amsthm}
\usepackage{amscd}

\usepackage[hidelinks]{hyperref}
\usepackage{geometry}

\newcommand{\setsymbol}[1]{\ensuremath{\mathbb{#1}}}%
\newcommand{\R}{\setsymbol{R}}%
\newcommand{\Sphere}{\setsymbol{S}}
\newcommand{\Torus}{\setsymbol{T}}%

\newcommand{\KulNo}{\ensuremath{{~\wedge\!\!\!\!\!\bigcirc~}}}

\newtheorem{definition}{Definition}[section]
\newtheorem{lemma}[definition]{Lemma}
\newtheorem{theorem}[definition]{Theorem}
\newtheorem{proposition}[definition]{Proposition}
\newtheorem{corollary}[definition]{Corollary}

\DeclareMathOperator{\Rm}{Rm}

\DeclareMathOperator{\Int}{Int}

\DeclareMathOperator{\trace}{tr}
\DeclareMathOperator{\linspan}{span}
\DeclareMathOperator{\Sym}{Sym}

\title{Preserving positive intermediate curvature}

\author{Tsz-Kiu Aaron Chow}
\address{Department of Mathematics \\ Massachusetts Institute of Technology \\ Cambridge, MA, 02139}
\email{chowtka@mit.edu}

\author{Florian Johne}
\address{Department of Mathematics \\ Columbia University \\ New York, NY, 10027}
\email{johne@math.columbia.edu}

\author{Jingbo Wan}
\address{Department of Mathematics \\ Columbia University \\ New York, NY, 10027}
\email{jw3976@columbia.edu}

\begin{document}

\begin{abstract}
 Consider a compact manifold $N$ (with or without boundary) of dimension $n$. Positive $m$-intermediate curvature interpolates between positive Ricci curvature ($m = 1$) and positive scalar curvature ($m = n-1$), and it is obstructed on partial tori $N^n = M^{n-m} \times \Torus^m$.
 Given Riemannian metrics $g, \bar{g}$ on $(N, \partial N)$ with positive $m$-intermediate curvature and 
 $m$-positive difference $h_g - h_{\bar{g}}$
 of second fundamental forms we show that there exists a smooth family of Riemannian metrics with positive $m$-intermediate curvature interpolating between $g$ and $\bar{g}$.
 Moreover, we apply this result to prove a non-existence result for partial torical bands with positive $m$-intermediate curvature and strictly $m$-convex boundaries.
\end{abstract}
\maketitle

\section{Introduction}
The existence of Riemannian metrics with positive curvature implies obstructions on the topology of closed manifolds:
Manifolds with topology $N^n = M^{n-1} \times \Sphere^1$ do not admit metrics of positive Ricci curvature by the Theorem of Bonnet--Myers,
while manifolds with topology $N^n = \Torus^n$ do not admit metrics of positive scalar curvature by the resolution of the Geroch conjecture
due to R.~Schoen and S.-T.~Yau 
\cite{Schoen-Yau:1979:Structure_PositiveScalar} (for $3 \leq n \leq 7$ by using minimal hypersurfaces) and M.~Gromov and H.-B.~Lawson
\cite{GromovLawson:1984:DiracOperator} (by using spinors and the Atyah--Singer index theorem).

The above results yield obstructions for positive curvature on the the partial tori $N^n = M^{n-m} \times \Torus^m$ for the limit cases $m = 1$ and $m = n-1$.
Recently, S.~Brendle, S.~Hirsch and the second author introduced the notion of positive intermediate curvature (see Section 2 for a precise definition), which interpolates between positive Ricci curvature $(m=1)$ and positive scalar curvature $(m = n-1)$. They obtained the following obstruction result on partial tori:
\begin{theorem}[Generalized Geroch conjecture, Theorem 1.5 in \cite{Brendle-Hirsch-Johne}] \ \\
Assume $n \leq 7$ and $1 \leq m \leq n-1$. Let $N^n$ be a closed and orientable manifold of dimension $n$,
 and suppose that there exists a closed and orientable manifold 
 $M^{n-m}$ and a map $F: N^n \rightarrow M^{n-m} \times \Torus^m$
 with non-zero degree.
 Then the manifold $N$ does not admit a metric with positive $m$-intermediate curvature.
 \label{Theorem:BHJ:GeneralizationGerochsConjecture}
\end{theorem}
The associated rigidity question was studied by J.~Chu, K-K.~Kwong and M.-C.~Lee \cite{Chu-Kwong-Lee:2022:Rigidity} in ambient dimension at most five. S.~Chen \cite{Chen:2022:ArbitraryEnds}
extended the obstruction result to manifolds with arbitrary ends. Moreover,
K.~Xu \cite{Xu:2023:DimensionConstraints} showed sharpness
of the result in \cite{Brendle-Hirsch-Johne} by constructing counterexamples for dimensions $n > 7$ and $3 \leq m \leq n-3$.
The above result was recently used by M.~L.~Labbi \cite{Labbi:2023:Preordering_PSC} to compute
the Riemann invariant of products of spheres and tori.

In this work, we study the interaction of the internal geometry and the boundary geometry for metrics of positive intermediate curvature.
The corresponding question for positive scalar curvature and mean curvature on the boundary dates back to work by 
M.~Gromov and H.-B.~Lawson \cite{GromovLawson:Spin}. A similar interaction appears in the proof of the positive mass theorem by R.~Schoen and S.-T.~Yau \cite{Schoen-Yau:1979:PMT} --- planes with positive mean curvature act as barriers for minimal hypersurfaces in the bulk region with positive scalar curvature. Related is work by Y.~Shi and L.-F.~Tam \cite{Shi-Tam:2002:BoundaryBehaviour},
where they proved an estimate for the integral of the mean curvature over the boundary in manifolds with non-negative scalar curvature.
P.~Miao \cite{Miao:2002:PMT-Corners} proved a positive mass theorem on manifolds with corners by performing a suitable intrinsic bending construction.
A pertubation argument, which made the boundary totally geodesic while keeping the scalar curvature non-decreasing, allowed S.~Brendle, F.C.~Marques and A.~Neves \cite{Brendle-Marques-Neves} to construct counterexamples to the Min-Oo conjecture.

Recently, the first author proved a general result 
\cite[Main Theorem 1]{Chow} on the interaction of internal geometry and boundary geometry by suitably gluing Riemannian metrics.
The result applies to a wide range of positive curvature conditions,
for example to metrics with positive curvature operator, PIC 2, PIC 1 (with convex boundary),
positive isotropic curvature (with two-convex boundary) and
positive scalar curvature (with mean-convex boundary).
For a different approach to the gluing problem, see also the thesis by A.~Schlichting,
\cite{Schlichting:2014:Thesis}.

The first result of our work extends the gluing result 
of the first author to $m$-intermediate curvature with the natural condition of $m$-convexity on the boundary:
\begin{theorem}
[Preserving positive $m$-intermediate curvature]
\label{main thm} \ \\
 Suppose that $N^n$ is a compact smooth manifold
 with smooth boundary $\partial N$ of dimension $\dim N = n$.
 Let $g, \tilde{g}$ be Riemannian metrics on $N$, such that
 $g = \tilde{g}$ on the boundary $\partial N$.

 Then there exists $\lambda_0 > 0$,
 a family of smooth Riemannian metrics
 $\{\hat{g}_{\lambda}\}_{\lambda > \lambda_0}$,
 and a neighborhood $U$ of the boundary $\partial N$,
 such that 
 the metric $\hat{g}_{\lambda}$ agrees with the metric $g$ outside of $U$,
 and the metric $\hat{g}_{\lambda}$ agrees with the metric
 $\tilde{g}$ in a neighbourhood of $\partial N$. Additionally,
 we have $\hat{g}_{\lambda} \rightarrow g$ as $\lambda \rightarrow \infty$ in $C^{\alpha}$ for any $\alpha \in (0,1)$.
 
 Moreover, let $1 \leq m \leq n-1$. If 
 \begin{enumerate}
  \item the Riemannian manifolds
  $(N,g)$ and $(N, \tilde{g})$ have positive $m$-intermediate curvature,
  \item the difference $h_g - h_{\tilde{g}}$ is strictly $m$-convex (i.e. strictly $m$-positive),
 \end{enumerate}
then the Riemannian manifold $(N, \hat{g}_{\lambda})$
has positive $m$-intermediate curvature for all $\lambda > \lambda_0$.
\end{theorem}

The main ingredient in the proof is Proposition \ref{Proposition:Convexity-LinearAlgebra} relating the cone of positive $m$-intermediate curvature to the Kulkarni--Nomizu product of $m$-convex symmetric two-tensors.

In the second part of the paper we consider Riemannian manifolds with $m$-positive intermediate curvature and strictly $m$-convex boundary:
Let us first recall the doubling lemma by M.~Gromov and H.-B.~Lawson for  manifolds with positive scalar curvature and strictly mean convex boundaries.
\begin{lemma}[Doubling of positive scalar curvature
metrics, M.~Gromov and H.-B.~Lawson \cite{GromovLawson:Spin}]\label{PSCdoubling} \ \\
Suppose $(N,g)$ is an orientable compact smooth Riemannian manifold
with smooth boundary $\partial N$. Assume the metric
$g$ has positive scalar curvature 
and is strictly mean convex (i.e.\ $H_{\partial N} > 0$) with respect
to the outward unit normal.
Then the double of $N$ carries a metric of positive scalar curvature.	
\end{lemma}

This lemma (in conjunction with the nonexistence of positive scalar curvature metrics on the torus) then implies the following obstruction to positive scalar curvature on torical bands:
\begin{theorem}[Boundaries of a torical band, M.~Gromov and H.-B.~Lawson, \cite{GromovLawson:Spin}]
\label{BoundaryToricalband} \ \\
	Consider
 the smooth manifold with boundary $N = \Torus^{n-1}\times [-1, 1]$ and let $g$ be a Riemannian metric on $N$ with positive scalar curvature. Then the boundary 
 $\partial N$ cannot be strictly mean convex.
\end{theorem}
D.~Räde extended the above result to a scalar curvature and mean curvature comparison result on more general bands \cite{Raede:2021:BandComparison}.

We extend the above result on scalar curvature, mean curvature and torical bands to partially torical bands
by proving a generalization of the doubling lemma
of M.~Gromov and H.-B.~Lawson.
\begin{theorem}[Boundaries of a partially torical band] \ \\
Let $n \leq 7$ and $1 \leq m \leq n-1$. Suppose $M^{n-m}$ is a closed orientable manifold. Consider the smooth manifold with boundary
$N = M^{n-m}\times \Torus^{m-1}\times [-1, 1]$ and a Riemannian metric $g$ with positive $m$-intermediate curvature on $N$. Then the boundary $\partial N$ cannot be strictly $m$-convex.
\label{boundarypartialtoricalband}
\end{theorem}

The work is structured as follows: 
In Section 2 we introduce our notation and recall the definition of intermediate curvature. In Section 3 we prove an algebraic lemma connecting the cone of positive $m$-intermediate curvature and $m$-convexity.
In Section 4 and 5 we prove the gluing result and in Section 6 we perform the doubling constructions.

\textbf{Acknowledgements:}  \ \\
The authors would like to thank Simon Brendle for discussions and encouragement, and the anonymous referee for careful reading of the article.

\section{Preliminaries}
\label{Section:Preliminaries}

Let $(V, \langle \cdot, \cdot \rangle)$ be a $n$-dimensional real inner product space.
The space of algebraic curvature tensors on $V$ denoted by $C_B(V)$
is given by multilinear maps $R: V \times V \times V \times V \rightarrow \R$ with the symmetries of the curvature tensor, i.e.\
\[
 R(v_1,v_2,v_3, v_4) = -R(v_2, v_1, v_3, v_4) \; \text{and} \;
 R(v_1, v_2, v_3, v_4) = R(v_3, v_4, v_1, v_2)
\]
for all $v_1, v_2, v_3, v_4 \in V$,
and satisfying the first Bianchi identity, i.e.\
\[
R(v_1, v_2, v_3, v_4) + R(v_3, v_1, v_2, v_4) + R(v_2, v_3, v_1, v_4) = 0
\]
for all $v_1, v_2, v_3, v_4 \in V$.

We denote the space of symmetric bilinear maps $T: V \times V \rightarrow \R$ by $\Sym^2 V$. The Kulkarni--Nomizu product $\KulNo: \Sym^2 V \times \Sym^2 V \rightarrow C_B(V)$ is defined by 
\[
(S \KulNo T)(v_1, v_2, v_3, v_4)
= S(v_1, v_3) T(v_2, v_4) + S(v_2, v_4) T(v_1, v_3)
- S(v_1, v_4) T(v_2, v_3) - S(v_2, v_3) T(v_1, v_4)
\]
for $v_1, v_2, v_3, v_4 \in V$.

Following Definition 1.1 in work of the S.~Brendle, S.~Hirsch and the second author
\cite{Brendle-Hirsch-Johne} we define the cone $\mathcal{C}_m(V)$ of non-negative $m$-intermediate curvature
in the space of algebraic curvature tensors by
\begin{align*}
 \mathcal{C}_m(V) :=
 \left\{
  T \in C_B(V)
 \; \left| \;
  \sum_{p=1}^m \sum_{q=p+1}^n T(e_p, e_q, e_p, e_q) \geq 0
  \; \text{for all orthonormal bases} \; \{e_i\}_{i=1}^n \; \text{of} \; V
  \right.
 \right\}.
\end{align*}

For a Riemannian manifold $(N^n,g)$ with boundary $\partial N$ we consider its Levi-Civita connection $D$ and its 
Riemann curvature tensor $\Rm_N$ given by the formula
\[
 \Rm_N(X,Y,Z,W)
 =
 -g( D_X D_Y Z - D_Y D_X Z - D_{[X,Y]} Z ,W)
\]
for vector fields $X,Y,Z,W \in \Gamma(TN)$. 

The Riemannian manifold $(N^n,g)$ has
positive $m$-intermediate curvature,
if $\Rm_N(p) \in \Int(\mathcal{C}_m)$ for all $p \in N$
(compare with Definition 1.1 in \cite{Brendle-Hirsch-Johne}).

Let $\nu$ be the inward pointing unit normal vector field on the boundary
$\partial N$. The scalar-valued second fundamental form $h_g: T(\partial N) \otimes T(\partial N) \rightarrow C^{\infty}(\partial N)$ 
of the boundary $\partial N$ with respect to the Riemannian metric
$g$ is defined by
\[
 h_g(X,Y) = g(\nu, D_X Y)
\]
for $X,Y \in \Gamma(TN)$.
With this convention the scalar-valued second fundamental form
is positive on the standard sphere $\Sphere^n \subset \R^{n+1}$
with respect to the inward pointing unit normal vector $\nu = -x$.

We say that the boundary $\partial N$ is strictly $m$-convex (where $1 \leq m \leq n-1$),
if the bilinear form $h_g(p)$ is $m$-positive for all $p \in \partial N$, i.e.\
if $\lambda_1 \leq \lambda_2 \leq \dots \lambda_{n-1}$
denote the eigenvalues of $h_g(p)$, then $\lambda_1 + \dots + \lambda_m > 0$.
For $m = 1$ we recover the notion of strict convexity, and for $m = n-1$ we recover 
the notion of strict mean convexity.

\section{Connecting $m$-convexity and $m$-intermediate curvature}

In this section we prove a lemma in linear algebra,
which allows us to connect the cone of positive $m$-intermediate curvature and strict $m$-convexity.
\begin{proposition}[$m$-intermediate curvature cone and $m$-convexity] \ \\
Let $(V, \langle \cdot, \cdot \rangle)$ be a $n$-dimensional inner product space and $S: V \times V \rightarrow \R$ be a symmetric bilinear form. Let $W \subset V$ be a $(n-1)$-dimensional subspace, and let $\nu \in W^{\perp}$ be a unit vector. Let $S|_W$ be the restriction of $S$ on $W$. Fix $1 \leq m \leq n -1$. Then the bilinear form $S|_W$ is $m$-positive, if and only if
\[
 S \KulNo (\nu^b \otimes \nu^b) \in \Int(\mathcal{C}_m(V)).
\]   
\label{Proposition:Convexity-LinearAlgebra}
\end{proposition}

\begin{proof} \ \\
Suppose that the bilinear form $S|_W$ is $m$-positive.
 We extend the bilinear form $S|_W$ to a bilinear $T$ on $V$
 by setting
 \[
  T(v,w) := S|_W (v^{\parallel}, w^{\parallel})
 \]
 for $v,w \in V$. Here $v^{\parallel}$
 denotes the orthogonal projection from $V$ to $W$.
 
 The assumption on the $m$-positivity of the bilinear form $S|_W$ on $W$
 implies that the bilinear form $T$ is $(m+1)$-positive on $V$.

We first want to show that $T$ being $(m+1)$-positive implies $T \KulNo (\nu^b \otimes \nu^b)\in \Int(\mathcal{C}_m(V))$.
 
 Let $\{e_1, \dots, e_n\}$ be an orthonormal basis
 of $V$ with respect to the inner product $\langle \cdot, \cdot \rangle$.
 We denote the components of the vector $\nu$ with respect to this orthonormal basis by $a_p$, i.e.\ $a_p = \langle \nu, e_p \rangle$.
 
 We have 
 \begin{align*}
  \left[ T \KulNo (\nu^b \otimes \nu^b) \right] (e_p, e_q, e_p, e_q)
  =
  a_p^2 T(e_q, e_q) + a_q^2 T(e_p, e_p) - 2 a_p a_q T(e_p, e_q)
 \end{align*}
 for $1 \leq p, q \leq n$ by definition of the Kulkarni--Nomizu product.
 We observe the identity
 \begin{align*}
  &2
  \sum_{p=1}^m \sum_{q = p+1}^n
    \left[ T \KulNo (\nu^b \otimes \nu^b) \right] (e_p, e_q, e_p, e_q) \\
 =&
 \left( 
  \sum_{p=1}^n \sum_{q=1}^n - \sum_{p=m+1}^{n} \sum_{q=m+1}^n
 \right)
   \left[ T \KulNo (\nu^b \otimes \nu^b) \right] (e_p, e_q, e_p, e_q).
 \end{align*}

 We evaluate the first term in the above sum:
 \begin{align*}
  &\sum_{p=1}^n \sum_{q=1}^n  
  \left[ T \KulNo (\nu^b \otimes \nu^b) \right] (e_p, e_q, e_p, e_q) \\
 =&
 2 \sum_{p=1}^n a_p^2 \sum_{q=1}^n T(e_q, e_q)
 - 2 T \left( \sum_{p=1}^n a_p e_p, \sum_{q=1}^n a_q e_q \right) 
 = 2 \trace_V(T) - 2 T(\nu, \nu) = 2 \trace_V(T).
 \end{align*}

 We evaluate the second term in the above sum:
 \begin{align*}
&\sum_{p=m+1}^{n} \sum_{q=m+1}^n
   \left[ T \KulNo (\nu^b \otimes \nu^b) \right] (e_p, e_q, e_p, e_q) \\
   =&
   2 \sum_{p=m+1}^n a_p^2 \sum_{q=m+1}^n T(e_q, e_q)
   - 2 T \left( 
    \sum_{p=m+1}^n a_p e_p, \sum_{q=m+1}^n a_q e_q
   \right).
 \end{align*}

 This implies
 \begin{equation}
  \begin{aligned}
  &\sum_{p=1}^m \sum_{q = p+1}^n
    \left[ T \KulNo (\nu^b \otimes \nu^b) \right] (e_p, e_q, e_p, e_q) \\
 =&
  \trace_V(T) \sum_{p=1}^m a_p^2
 +
  \sum_{p=m+1}^n a_p^2 \sum_{q=1}^m T(e_q, e_q)
 +  T \left( 
    \sum_{p=m+1}^n a_p e_p, \sum_{q=m+1}^n a_q e_q
   \right).
 \end{aligned}
 \label{Section2:Identity:Kulkarni-Nomizu}
 \end{equation}

 If $a_p = 0$ for all $m+1 \leq p \leq n$, then we deduce the estimate
 \begin{align*}
  \sum_{p=1}^m \sum_{q = p+1}^n
    \left[ T \KulNo (\nu^b \otimes \nu^b) \right] (e_p, e_q, e_p, e_q) 
 =  \trace_V(T) \sum_{p=1}^m a_p^2 > 0,
 \end{align*}
 since the bilinear form $T$ is $(m+1)$-positive by construction and hence $n$-positive.
 Hence we deduce $T \KulNo (\nu^b \otimes \nu^b) \in \Int(\mathcal{C}_m)$ in this case.
 
 Now suppose that $a_p \neq 0$ for some $m + 1 \leq p \leq n$.
 We define the unit vector
 \[
  w := \left( \sum_{p=m+1}^{n} a_p^2 \right)^{-\frac{1}{2}}
  \sum_{q=m+1}^n a_q e_q.
 \]
 With this definition we deduce from equation \eqref{Section2:Identity:Kulkarni-Nomizu}
 the identity
  \begin{align*}
  &\sum_{p=1}^m \sum_{q = p+1}^n
    \left[ T \KulNo (\nu^b \otimes \nu^b) \right] (e_p, e_q, e_p, e_q) 
 =
  \trace_V(T) \sum_{p=1}^m a_p^2
 +
  \sum_{p=m+1}^n a_p^2
 \left( \sum_{q=1}^m T(e_q, e_q) + T(w,w) \right).
 \end{align*}
 
 The first term involving the trace $\trace_V(T)$ is positive as above.
 The term in the bracket is positive, since the bilinear form
 $T$ is $(m+1)$-positive,
 and $w \perp \linspan\{e_1, \dots, e_m\}$
 by construction.
 Hence the sum is positive and we deduce $T \KulNo (\nu^b \otimes \nu^b)
 \in \Int(\mathcal{C}_m(V))$.  
 
 On the other hand, by the construction of $T$, the restriction of $S-T$  to the subspace $W$ vanishes. Therefore, we may write
 \[	S = T + \omega\otimes \nu^b+ \nu^b \otimes \omega,\]
where $\omega$ is a suitable 1-form. 
Note that
\[	(\omega\otimes \nu^b + \nu^b \otimes \omega)\KulNo (\nu^b\otimes \nu^b) = 0 \] 
by symmetry. Hence,
\[ S\KulNo ( \nu^b\otimes \nu^b)\in \Int(\mathcal{C}_m(V)). \]

 The other implication in the equivalence follows by taking
 the orthonormal basis $\{f_1, \dots, f_{n-1}, \nu \}$ of the vector space $V$,
 where $\{f_1, \dots, f_{n-1}\}$ is an orthonormal basis of the subspace $W$.
\end{proof}

\section{Preserving Curvature Conditions}

In this section, we assume that $g$ and $\tilde{g}$ are Riemannian metrics on $N$ such that $g - \tilde{g} = 0$ along $\partial N$. We describe our choice of perturbation as in 
work of S.~Brendle, F.C.~Marques and A.~Neves \cite{Brendle-Marques-Neves}.
We fix a neighborhood $U$ of the boundary $\partial N$ and a smooth boundary defining function $\rho: N\to [0,\infty)$ by taking it to be the distance function from the boundary $\partial N$ with respect to the metric $g$. Then we have $|D\rho|_g =1$. Since $g-\tilde{g}=0$ along the boundary $\partial N$, we can find a symmetric (0,2)-tensor $S$ such that  $\tilde{g}=g+\rho S$ in a neighborhood of $\partial N$ and $S=0$ outside $U$. The scalar-valued second fundamental forms and the boundary defining function satisfy
\begin{align*}
    \frac{1}{2}S(X,Y)= h_g(X,Y)-h_{\tilde{g}}(X,Y),
    \; \text{and} \;
    D^2\rho(X,Y) = -h_g(X,Y).\notag
\end{align*}
for all $X,Y \in \Gamma(T (\partial N))$. 
This implies that the identity
\begin{align}
    &h_g(X,Y) - h_{\tilde{g}}(X,Y)=\frac{1}{2}S(X,Y)=-D^2\rho(X,Y) - h_{\tilde{g}}(X,Y)
\end{align}
holds on the boundary $\partial N$ for all $X,Y\in \Gamma(T(\partial N))$.

We choose a smooth cut-off function $\chi:[0,\infty)\to[0,1]$ with the following properties (compare with \cite[Lemma 17]{Brendle-Marques-Neves}):
\begin{itemize}
    \item $\chi(s)=s-\frac{1}{2}s^2$ for $s\in[0,\frac{1}{2}]$;
    \item $\chi(s)$ is constant for $s\geq 1$;
    \item $\chi''(s)<0$ for $s\in[0,1)$.
\end{itemize}

Moreover, we choose a smooth cut-off function $\beta:(-\infty,0]\to[0,1]$ such that
\begin{itemize}
    \item $\beta(s)=\frac{1}{2}$ for $s\in [-1,0]$;
    \item $\beta(s)=0$ for $s\in (-\infty,-2]$.
\end{itemize}

For $\lambda > 0$ sufficiently large we define a smooth metric $\hat{g}_{\lambda}$ on 
the manifold $N$ by the formula 
\begin{align}
    \hat{g}_{\lambda}=
    \begin{cases}
    g+\lambda^{-1}\chi(\lambda\rho)S\quad\quad\quad\quad &\text{for}\quad \rho\geq e^{-\lambda^2}\\
    \tilde{g}-\lambda\rho^2\beta(\lambda^{-2}\log\rho)S\quad &\text{for}\quad\rho<e^{-\lambda^2}.
    \end{cases}
\end{align}
In the sequel, we will show that $\hat{g}_{\lambda}$ preserves 
positive $m$-intermediate curvature of $g$ and $\tilde{g}$ for sufficiently large $\lambda > 0$. Note that we have the identity $\hat{g}_{\lambda}=\tilde{g}$ in the region $\{\rho\leq e^{-2\lambda^2}\}$ and $\hat{g}_{\lambda}=g$ outside the neighbourhood $U$. Moreover, from the construction it follows that  $\hat{g}_{\lambda} \rightarrow g$ as $\lambda \rightarrow \infty$ in $C^{\alpha}$ for any $\alpha \in (0,1)$.

We first derive a lower bound for the $m$-intermediate curvature of the metric $\hat{g}_{\lambda}$. We first consider the region $\{\rho\geq e^{-\lambda^2}\}$.

\begin{proposition}[Curvature estimates
in inner gluing region] \ \\
 Suppose that $h_g - h_{\tilde{g}}$ is $m$-positive on the boundary $\partial N$.
  Let $\epsilon > 0$ be given. If $\lambda=\lambda (\epsilon, \chi)>0$ is sufficiently large, then 
\begin{align*}
    \sum_{p=1}^m\sum_{q=p+1}^n (\Rm_{\hat{g}_{\lambda}}(x)(e_p, e_q, e_p, e_q)
    -\Rm_{g}(x)(e_p, e_q, e_p, e_q))\geq -\epsilon
\end{align*}
for any $\hat{g}_{\lambda}$-orthonormal frame $\{e_1, \dots, e_n\}$ and any $x \in N$ in the region $\{\rho(x)\geq e^{-\lambda^2}\}$. 
\label{Proposition:CurvatureEstimates_InnerRegion}
\end{proposition}

\begin{proof} \ \\
We fix a point $x \in N$ such that $\rho(x)\geq e^{-\lambda^2}$. Let $\{e_1, \dots, e_n\}$ be a geodesic normal frame around the point $x$ with respect to the metric $\hat{g}_{\lambda}$.  Let $\varphi$ be a two-form.
We write $\varphi=\sum_{i,j}\varphi^{ij}e_i\wedge e_j$ for coefficients $\varphi^{ij}$, which are anti-symmetric in $i$ and $j$. 
In the following the Einstein summation convention will be adopted freely. Since $\varphi$ is in particular a (2,0)-tensor, $\varphi$ induces by the fundamental principle
of tensor calculus a linear map $[\varphi]: (T_x N)^*\to T_x N$ via the action $[\varphi]w := \varphi^{ij}w(e_i)e_j$.
Equation (5) in work of the first author \cite{Chow} yields 
the estimate
\begin{align*} 
\Rm_{\hat{g}_{\lambda}}(\varphi, \varphi) - \Rm_g(\varphi, \varphi) 
&\notag\geq 2 \lambda (-\chi''(\lambda\rho)) S([\varphi] d\rho, [\varphi] d\rho) \\ 
&\notag\quad+ \varphi^{ij} \varphi^{kl} \chi'(\lambda\rho) (-2 D_i D_k \rho \, S_{jl} - \frac{1}{2} \chi'(\lambda\rho) |D\rho|_{\hat{g}_\lambda}^2 S_{ik} S_{jl}) \\ 
&\notag\quad- C \chi'(\lambda\rho) |[\varphi] d\rho| - C \lambda^{-1},
\end{align*} 
where $C>0$ is a positive constant independent of $\lambda$.

To pick up the sectional curvatures, we choose for $1 \leq p,q \leq n$ the two-form $\varphi_{pq}$ by specifying the components $(\varphi_{pq})^{ij} = \delta_p^i\delta_q^j - \delta_p^j\delta_q^i$. This implies the identities  
\begin{align*}
   [\varphi_{pq}]d\rho = \nabla_p\rho\ e_q - \nabla_q\rho\ e_p\quad\text{and}\quad 4\Rm(e_p, e_q, e_p, e_q) = \Rm(\varphi_{pq}, \varphi_{pq}).
\end{align*}
Recall that $S$ is supported in $U$, so $D\rho$ and $D^2\rho$ are uniformly bounded with respect to the metric $\hat{g}_{\lambda}$. Thus the above estimate implies after summation the inequality
\begin{equation}
\begin{aligned}
	&\sum_{p=1}^m\sum_{q=p+1}^n( \Rm_{\hat{g}_{\lambda}}(e_p, e_q, e_p, e_q) - \Rm(e_p, e_q, e_p, e_q))\\
	 \geq& \frac{1}{2} \lambda (-\chi''(\lambda\rho))\ \sum_{p=1}^m\sum_{q=p+1}^n S\Big(\nabla_p\rho\ e_q - \nabla_q\rho\ e_p,\ \nabla_p\rho\ e_q - \nabla_q\rho\ e_p\Big) 
  - C\chi'(\lambda\rho) - C\lambda^{-1}
 \end{aligned}
 {\label{perturbation est1}}
\end{equation}
in the region $\{\rho\geq e^{-\lambda^2}\}$. Here the constant $C$ is independent of $\lambda$, but it does depend on $S, g, \rho, N$.

By our assumption on the scalar-valued second fundamental forms and Proposition \ref{Proposition:Convexity-LinearAlgebra} we deduce 
\[
S\KulNo ( d\rho\otimes d\rho)\in \Int(\mathcal{C}_m)
\]
(with respect to the metric $g$) at each point on $\partial N$.
This allows us to fix a small number $a>0$ such that 
\[
S\KulNo ( d\rho\otimes d\rho) - 4a|\nabla\rho|^2\delta\KulNo\delta  \in \mathcal{C}_m
\]
 (with respect to the metric $g$) in a small neighborhood of $\partial N$  where $\rho \geq e^{-\lambda^2}$. 
 From the construction of the metric $\hat{g}_{\lambda}$ we deduce
\[
S\KulNo ( d\rho\otimes d\rho) - 2a|\nabla\rho|^2\delta\KulNo\delta \in \mathcal{C}_m
\]
 (with respect to the metric $\hat{g}_{\lambda}$) in a small neighborhood of $\partial N$ where $\rho \geq e^{-\lambda^2}$. 
  Moreover, we observe that
 \[
 S\Big(\nabla_p\rho\ e_q - \nabla_q\rho\ e_p,\ \nabla_p\rho\ e_q - \nabla_q\rho\ e_p\Big) = S\KulNo ( d\rho\otimes d\rho)(e_p, e_q, e_p, e_q).
 \]
For a positive constant $C(n,m):=\sum_{p=1}^m\sum_{q=p+1}^n(\delta\KulNo\delta)(e_p, e_q, e_p, e_q)$ we observe the estimate
 \begin{equation}{\label{perturbation est2}}
 \begin{aligned}
 \sum_{p=1}^m\sum_{q=p+1}^n S\Big(\nabla_p\rho\ e_q - \nabla_q\rho\ e_p,\ \nabla_p\rho\ e_q - \nabla_q\rho\ e_p\Big)
 &= \sum_{p=1}^m\sum_{q=p+1}^n (S\KulNo ( d\rho\otimes d\rho))(e_p, e_q, e_p, e_q) \\
 &\geq 2a C(n,m)|\nabla\rho|^2
 \end{aligned}
 \end{equation}
in a small neighborhood of the boundary $\partial N$  where $\rho \geq e^{-\lambda^2}$. Combining 
the estimate (\ref{perturbation est1}) and the estimate (\ref{perturbation est2}) we obtain the estimate
\begin{equation}{\label{perturbation est3}}
\begin{aligned}
	&\sum_{p=1}^m\sum_{q=p+1}^n( \Rm_{\hat{g}_{\lambda}}(e_p, e_q, e_p, e_q) - \Rm(e_p, e_q, e_p, e_q)) \\
	\geq& a C(n,m) \lambda (-\chi''(\lambda\rho)) |\nabla\rho|^2  - C\chi'(\lambda\rho) - C\lambda^{-1}
 \end{aligned}
\end{equation}
in the region $\{\rho\geq e^{-\lambda^2}\}$.

We split the region into two sub-regions as follows. Let us fix a real number $s_0\in [0,1)$ such that $C\chi'(s_0) < \frac{\epsilon}{2}$. By the construction of the cut-off function $\chi$, we have $\inf_{0\leq s\leq s_0} (-\chi''(s)) > 0$. This implies in the region $\{e^{-\lambda^2}\leq\rho < s_0\lambda^{-1}\}$ the estimate
\begin{align*}
	&\inf_{e^{-\lambda^2}\leq\rho	<s_0\lambda^{-1}}\left(\sum_{p=1}^m\sum_{q=p+1}^n( \Rm_{\hat{g}_{\lambda}}(e_p, e_q, e_p, e_q) - \Rm(e_p, e_q, e_p, e_q))\right) \\
	\geq& a C(n,m) \lambda\, \inf_{e^{-\lambda^2}\leq\rho	<s_0\lambda^{-1}}\left( (-\chi''(\lambda\rho)) |\nabla\rho|^2\right)  - C - C\lambda^{-1}\\
 \geq& a C(n,m) \lambda\, \inf_{0\leq s\leq s_0} (-\chi''(s))  - C - C\lambda^{-1}.
\end{align*}
Thus, we obtain
\[	\inf_{e^{-\lambda^2}\leq\rho	<s_0\lambda^{-1}}\left(\sum_{p=1}^m\sum_{q=p+1}^n( \Rm_{\hat{g}_{\lambda}}(e_p, e_q, e_p, e_q) - \Rm(e_p, e_q, e_p, e_q))\right)\to \infty\]
as $\lambda\to\infty$. Moreover, in the region $\{ \rho\geq s_0\lambda^{-1}\}$ we have
\begin{align*}
	&\inf_{\rho	\geq s_0\lambda^{-1}}\left(\sum_{p=1}^m\sum_{q=p+1}^n( \Rm_{\hat{g}_{\lambda}}(e_p, e_q, e_p, e_q) - \Rm(e_p, e_q, e_p, e_q))\right)
	\geq -C\chi'(s_0) - C\lambda^{-1}.
\end{align*}
Since $C\chi'(s_0) < \frac{\epsilon}{2}$, it follows that
\[	\inf_{\rho	\geq s_0\lambda^{-1}}\left(\sum_{p=1}^m\sum_{q=p+1}^n( \Rm_{\hat{g}_{\lambda}}(e_p, e_q, e_p, e_q) - \Rm(e_p, e_q, e_p, e_q))\right) \geq -\epsilon\]
if $\lambda > 0$ is sufficiently large. Putting the above together, we conclude that
\[
\inf_{\rho	\geq e^{-\lambda^2}}\left(\sum_{p=1}^m\sum_{q=p+1}^n( \Rm_{\hat{g}_{\lambda}}(e_p, e_q, e_p, e_q) - \Rm(e_p, e_q, e_p, e_q))\right) \geq -\epsilon\]
if $\lambda > 0$ is sufficiently large. This completes the proof
of Proposition \ref{Proposition:CurvatureEstimates_InnerRegion}.
\end{proof}

We next consider the region $\{\rho<e^{-\lambda^2}\}$:
\begin{proposition}
[Curvature estimates in outer gluing region]
\ \\
 Suppose that $h_g - h_{\tilde{g}}$ is $m$-positive on the boundary $\partial N$.  
  Let $\epsilon > 0$ be an arbitrary positive real number. If $\lambda = \lambda(\epsilon, \beta) >0$ is sufficiently large, then 
\begin{align*}
    \sum_{p=1}^m\sum_{q=p+1}^n (\Rm_{\hat{g}_{\lambda}}(x)(e_p, e_q, e_p, e_q)
    -\Rm_{\tilde{g}}(x)(e_p, e_q, e_p, e_q)) \geq -\epsilon
\end{align*}
for any $\hat{g}_{\lambda}$-orthonormal frame $\{e_1, \dots, e_n\}$
and any $x\in N$ in the region $\{\rho(x)< e^{-\lambda^2}\}$. 
\label{Proposition:CurvatureEstimates_OuterRegion}
\end{proposition}

\begin{proof} \ \\
In the region $\{\rho<e^{-\lambda^2}\}$, we have $\hat{g}_{\lambda}=\tilde{g}+\tilde{h}_{\lambda}$, where $\tilde{h}_{\lambda}$ is defined by 
\begin{align*}
    \tilde{h}_{\lambda} = -\lambda\rho^2\beta(\lambda^{-2}\log\rho)S.
\end{align*}

Let $\{e_1, \dots, e_n\}$ be a geodesic normal frame around $x$ with respect to the metric $\hat{g}_{\lambda}$. Equation (12) in 
work of the first author \cite{Chow} implies 
\begin{align}{\label{perturbation est4}}
    \Rm_{\hat{g}_{\lambda}}(\varphi, \varphi)-\Rm_{\tilde{g}}(\varphi, \varphi) \geq\ &  2\lambda\beta(\lambda^{-2}\log\rho)S([\varphi]d\rho, [\varphi]d\rho) - L\lambda^{-1}
\end{align}
for any two-form $\varphi = \varphi^{ij}e_i\wedge e_j$
and a positive constant $L > 0$ independent of $\lambda$. Choosing
for $1 \leq p,q \leq n$
the two-form $\varphi_{pq}$ by
$(\varphi_{pq})^{ij} = \delta_p^i\delta_q^j - \delta_p^j\delta_q^i$ in equation (\ref{perturbation est4}) and summing over $p,q$, we obtain
\begin{align*}
	&\sum_{p=1}^m\sum_{q=p+1}^n( \Rm_{\hat{g}_{\lambda}}(e_p, e_q, e_p, e_q) - \Rm_{\tilde{g}}(e_p, e_q, e_p, e_q))	\\
	&\notag\geq \frac{1}{2}\lambda\beta(\lambda^{-2}\log\rho) \sum_{p=1}^m\sum_{q=p+1}^n S\Big(\nabla_p\rho\ e_q - \nabla_q\rho\ e_p,\ \nabla_p\rho\ e_q - \nabla_q\rho\ e_p\Big) - L\lambda^{-1}.
\end{align*}
Proceeding similarly as in the proof of Proposition \ref{Proposition:CurvatureEstimates_InnerRegion}, we have
\begin{align*}
\sum_{p=1}^m\sum_{q=p+1}^n S\Big(\nabla_p\rho\ e_q - \nabla_q\rho\ e_p,\ \nabla_p\rho\ e_q - \nabla_q\rho\ e_p\Big) \geq 2a|\nabla\rho|_{\tilde{g}}^2.	
\end{align*}
(with respect to the metric $\hat{g}_{\lambda}$) in a neighborhood of $\partial N$  where $\rho < e^{-\lambda^2}$. This implies
\begin{align}{\label{perturbation est5}}
	&\sum_{p=1}^m\sum_{q=p+1}^n( \Rm_{\hat{g}_{\lambda}}(e_p, e_q, e_p, e_q) - \Rm_{\tilde{g}}(e_p, e_q, e_p, e_q))	
	\geq a\lambda\beta(\lambda^{-2}\log\rho) |\nabla\rho|^2 - L\lambda^{-1}	
\end{align}
in the region $\{\rho<e^{-\lambda^2}\}$. Hence, if $\lambda>0$ is sufficiently large, then we have
\begin{align*}
    \inf_{\rho<e^{-\lambda^2}}\left(\sum_{p=1}^m\sum_{q=p+1}^n( \Rm_{\hat{g}_{\lambda}}(e_p, e_q, e_p, e_q) - \Rm_{\tilde{g}}(e_p, e_q, e_p, e_q))\right) \geq -\epsilon.
\end{align*}
From this, the assertion follows.
\end{proof}

Combining Proposition \ref{Proposition:CurvatureEstimates_InnerRegion} and Proposition \ref{Proposition:CurvatureEstimates_OuterRegion}, we can summarize the results in this section:
\begin{corollary}{\label{curv est cor}} \ \\
 Suppose that $h_g - h_{\tilde{g}}$ is $m$-positive on the boundary $\partial N$. Let $\epsilon>0$ be an arbitrary positive real number. If $\lambda = \lambda(\epsilon, \chi, \beta) >0$ is sufficiently large, then we have the pointwise inequality	
	\begin{align*}
    &\sum_{p=1}^m\sum_{q=p+1}^n \Rm_{\hat{g}_{\lambda}}(x)(e_p, e_q, e_p, e_q) \\
    \geq& \min\left\{  \sum_{p=1}^m\sum_{q=p+1}^n \Rm_g(x)(e_p, e_q, e_p, e_q),  \sum_{p=1}^m\sum_{q=p+1}^n \Rm_{\tilde{g}}(x)(e_p, e_q, e_p, e_q)\right\} - \epsilon
	\end{align*}
	for any $\hat{g}_{\lambda}$-orthonormal frame $\{e_1, \dots, e_n\}$ and any $x\in N$.
\end{corollary}

\section{Proof of Theorem \ref{main thm} on preserving positive intermediate curvature}

 Suppose that $h_g - h_{\tilde{g}}$ is $m$-positive on the boundary $\partial N$. 
  Suppose also that $(N,g)$ and $(N, \tilde{g})$ have positive $m$-intermediate curvature.
  
  Fix a point $x \in N$ and let $\{E_1, \dots, E_N\}$ be an orthonormal basis
  of the tangent space $T_x N$ with respect to the Riemannian metric $\hat{g}_{\lambda}$. 
   We want to show that
   \[	\sum_{p=1}^m\sum_{q=p+1}^n \Rm_{\hat{g}_{\lambda}}(E_p, E_q, E_p, E_q)  > 0.\]
  We divide the proof into two cases: 
  In the first case $x \in N$ is in the inner gluing region $\{\rho\geq e^{-\lambda^2}\}$; 
  and in the second case $x \in N$ is in the outer gluing region $\{\rho < e^{-\lambda^2}\}$.

For the first case, we have $g = \hat{g}_{\lambda} - h_{\lambda}$ where $h_{\lambda}(x)=\lambda^{-1}\chi(\lambda\rho)S(x)$. Fix a point $x\in N$. We evolve the orthonormal basis $\{E_1, \dots, E_n\}$ in the tangent space $T_x N$ by the linear ordinary differential equation
\begin{align}
    \begin{cases}
    \frac{d}{ds}E_i(s) &=\ \frac{1}{2}h_{\lambda} \circ E_i(s)\\
    E_i(0)&=\ E_i.
    \end{cases}
\end{align}
Then we see that the basis $\{E_i(s)\}$ remains orthonormal with respect to the metric $g_s= \hat{g}_{\lambda}-sh_{\lambda}$. We define $e_i := E_i(1)$. Then the basis $\{e_1, \dots, e_n\}$ is orthonormal with respect to the metric $g = \hat{g}_{\lambda} - h_{\lambda}$. It follows that
 \[	\sum_{p=1}^m\sum_{q=p+1}^n \Rm_{g}(e_p, e_q, e_p, e_q)  > 0.\]
In view of Corollary \ref{curv est cor}, it suffices to show that  
\[	\sum_{p=1}^m\sum_{q=p+1}^n \Rm_{g}(E_p, E_q, E_p, E_q)  > 0.\]
By writing $E_i(s)=A_i^k(s)E_k$, we observe that
\begin{align*}
    \frac{d}{ds}E_i(s) &= \frac{d}{ds}A_i^k(s)E_k
    = \frac{1}{2}h_l^kA_i^l(s)E_k
    = \frac{1}{2}\lambda^{-1}\chi(\lambda\rho)S_l^kA_i^l(s)E_k
\end{align*}
for any $s\in[0,1]$. This implies the estimate
\begin{align*}
    &\Big|\frac{d}{ds}\Rm_g(E_p(s), E_q(s), E_p(s), E_q(s))\Big| \leq C\lambda^{-1},
\end{align*}
where $C$ is a positive constant depending only on $N,g,\tilde{g}$ and $\chi$. This implies that
\begin{align*}
    \Big|\Rm_g(E_p, E_q, E_p, E_q) - \Rm_g(e_p, e_q, e_p, e_q)\Big|
    &\leq \int_0^1\frac{C}{\lambda}d\tau\leq \frac{C}{\lambda}. \notag
\end{align*}
Hence we have
\begin{align*}
    \sum_{p=1}^m\sum_{q=p+1}^n \Rm_{g}(E_p, E_q, E_p, E_q) \geq \frac{1}{2} \sum_{p=1}^m\sum_{q=p+1}^n \Rm_{g}(e_p, e_q, e_p, e_q) > 0,
\end{align*}
if $\lambda>0$ is sufficiently large.

For the second case, we have $\tilde{g} = \hat{g}_{\lambda} +  \lambda\rho^2\beta(\lambda^{-2}\log\rho)S$. Thus $h_{\lambda} = \lambda\rho^2\beta(\lambda^{-2}\log\rho)S$. Following the same argument as in the first case and using the fact that 
\begin{align*}
    |h_{\lambda}|\leq C(g,\tilde{g},\beta)e^{-\lambda^2}\leq \frac{C(g,\tilde{g},\beta)}{\lambda}
\end{align*}
in the region $\{\rho<e^{-\lambda^2}\}$ for sufficiently large $\lambda>0$, we also deduce
\[	\sum_{p=1}^m\sum_{q=p+1}^n \Rm_{\tilde{g}}(E_p, E_q, E_p, E_q) \geq \frac{1}{2} \sum_{p=1}^m\sum_{q=p+1}^n \Rm_{\tilde{g}}(\tilde{e}_p, \tilde{e}_q, \tilde{e}_p, \tilde{e}_q) > 0\]
for sufficiently large $\lambda > 0$. 

Combining the two cases together,  Corollary \ref{curv est cor} implies that 
\begin{align*}
   \sum_{p=1}^m\sum_{q=p+1}^n \Rm_{\hat{g}_{\lambda}}(E_p, E_q, E_p, E_q) >0
\end{align*}
for sufficiently large $\lambda>0$. Since the point $x\in N$ and 
the orthonormal basis $\{E_1,...,E_n\}\subset T_x N$ are arbitrary, we conclude that
\[
\Rm_{\hat{g}_{\lambda}}\in \Int(\mathcal{C}_m)
\]
on $N$ for sufficiently large $\lambda > 0$.
This finishes the proof of the theorem.

\section{Proof of Theorem \ref{boundarypartialtoricalband}
on boundaries of partially torical bands}

In this section we prove Theorem \ref{boundarypartialtoricalband} by using the strategy outlined by M.~Gromov and H.-B.~Lawson.
We prove a generalization of their doubling lemma 
(as stated in Lemma \ref{PSCdoubling}).

Once we have established the doubling lemma
Theorem \ref{boundarypartialtoricalband} follows
directly by the generalization of the Geroch conjecture,
Theorem \ref{Theorem:BHJ:GeneralizationGerochsConjecture}.

\begin{lemma}[Doubling for positive $m-$intermediate curvature]\label{doublingk} \ \\
Suppose $(N^n,g)$ is an orientable compact smooth Riemannian manifold
with smooth boundary $\partial N$, such that
the metric $g$ has positive $m$-intermediate curvature,
and the boundary $\partial N$ is strictly $m$-convex (possibly finitely many connected components).
Then the double of $N$ carries a metric of
positive $m$-intermediate curvature.
\end{lemma}

 We mimic the doubling process by M.~Gromov and H.-B.~Lawson, see Theorem 5.7 in \cite{GromovLawson:Spin}. We will closely follow their notations and constructions. 
 For completeness we describe their construction in detail.
 
We first shrink the manifold $N$ a little bit while preserving its boundary condition as follow: let $N_1=N\backslash C$ where $C$ is a thin collar of the boundary $\partial N$ and $N_1$ is chosen so that $\partial N_1$ is still strictly $m$-mean convex. We then consider the Riemannian product $N\times I$ with Riemannian metric $g_N+dt\otimes dt$ and define $D(N) = \{p \in  N\times I \mid \text{dist}(p, N_1) = \epsilon\}$, where $0 <\epsilon \ll 1$. Note that $D(N)$ is homeomorphic to the double of $N$.

Now we fix a point $x\in \partial N_1$, and let $\sigma$ be the geodesic segment in $N_1$ emanating orthogonally from $\partial N_1$ at $x$. Then the product $\sigma\times I$ will be totally geodesic in the product $N\times I$.

Let $\mu_1,...\mu_{n-1}$ be the principal curvatures of $\partial N_1$ at $x$.
By the construction of M.~Gromov and H.-B.~Lawson the principle
curvatures of $D(N)$ will be of
the following form for a point $x_{\theta}$ corresponding
to an angle $\theta$ (see Figure 8 in \cite{GromovLawson:Spin}): 
\begin{align}\label{principalcurvatures}
	\lambda_k=(\mu_k+O(\epsilon))\cos \theta +O(\epsilon^2 ) \; \text{for} \; k = 1, \dots, n-1, \; \text{and} \;\lambda_n=\frac{1}{\epsilon}\cos \theta +O(\epsilon).
\end{align}

As in the Figure 8 in \cite{GromovLawson:Spin}, we have a natural polar coordinates describing these $x_\theta$, let us denote $pr_N(x_\theta)$ the projection of $x_\theta$ to the corresponding point on $\partial N_1$.  Since the bilinear forms $h_{\partial N_1}$ at $pr_N(x_\theta)$ and $h_{D(N)}$ at $x_{\theta}$ are diagonalized simultaneously in this construction, we have the following relation:
\begin{align*}
	h_{D(N)}\Big|_{x_\theta}=\left((\cos\theta) h_{\partial N_1}\Big|_{pr_N(x_\theta)} +O(\epsilon)\right)\oplus\left(\frac{1}{\epsilon}\cos \theta +O(\epsilon)\right)(\nu^\flat\otimes \nu^\flat )\Big|_{pr_N(x_\theta)}.
\end{align*}

We apply the Gauss equation to $D(N)$ as a submanifold of $N\times I$ at the point $x_\theta$: We have for any orthonormal basis $\{e_i\}_{i=1}^n$ at the point $x_\theta$ the relation
\begin{align}\label{Gausseqn}
\text{Rm}^{D(N)}(e_i,e_j,e_i,e_j)	= \text{Rm}^{N\times I}(e_i,e_j,e_i,e_j) +h_{D(N)}(e_i,e_i)h_{D(N)}(e_j,e_j)-h_{D(N)}(e_i,e_j)^2.
\end{align}

We note that the second fundamental form terms
of $D(N)$ are related to the second fundamental form terms
of the boundary $\partial N_1$. This implies
\begin{align*}
	&h_{D(N)}(e_i,e_i)h_{D(x)}(e_j,e_j)-h_{D(N)}(e_i,e_j)^2\Big|_{x_\theta}\\
	=&\left((\cos\theta) h_{\partial N_1}(e_i,e_i)+\frac{1}{\epsilon}\cos\theta\cdot \nu^\flat(e_i)^2 +O(\epsilon)\right)\Big|_{pr_N(x_\theta)} \\
 &\cdot \left((\cos\theta) h_{\partial N_1}(e_j,e_j)+\frac{1}{\epsilon}\cos\theta\cdot \nu^\flat(e_j)^2 +O(\epsilon)\right)\Big|_{pr_N(x_\theta)}\\
	&-\left((\cos\theta) h_{\partial N_1}(e_i,e_j)+\frac{1}{\epsilon}\cos\theta\cdot \nu^\flat(e_i)\nu^\flat(e_j) +O(\epsilon)\right)^2\Big|_{pr_N(x_\theta)}\\
	=&\frac{1}{\epsilon^2}\left(\cos^2\theta \nu^\flat(e_i)^2\nu^\flat(e_j)^2- \cos^2\theta \nu^\flat(e_i)^2\nu^\flat(e_j)^2 \right)\Big|_{pr_N(x_\theta)}\\
	&+\frac{1}{\epsilon}\cos^2\theta \left(h_{\partial N_1}(e_i,e_i)\nu^\flat (e_j)^2+h_{\partial N_1}(e_j,e_j)\nu^\flat (e_i)^2-2h_{\partial N_1}(e_i,e_j)\nu^\flat (e_i)\nu^\flat (e_j)\right)\Big|_{pr_N(x_\theta)}\\
	&+O(1)\\
	=&\frac{1}{\epsilon}\cos^2\theta \left(h_{\partial N_1} \KulNo (\nu^\flat\otimes \nu^\flat ) \right)(e_i,e_j,e_i,e_j)\Big|_{pr_N(x_\theta)}+O(1),
\end{align*}
where the terms of order $O(\frac{1}{\epsilon^2})$ cancelled.

Summation of equation (\ref{Gausseqn}) yields the following identity for the $m$-intermediate curvature:
\begin{align*}
&\sum\limits_{p=1}^m\sum\limits_{q=p+1}^n \text{Rm}^{D(N)}(e_p,e_q,e_p,e_q )\Big|_{x_\theta}\\
	=&\sum\limits_{p=1}^m\sum\limits_{q=p+1}^n \left( \text{Rm}^{N\times I}(e_p,e_q,e_p,e_q) +h_{D(N)}(e_p,e_p)h_{D(N)}(e_q,e_q)-h_{D(N)}(e_p,e_q)^2\right)\Big|_{x_\theta}\\
    =&\sum\limits_{p=1}^m\sum\limits_{q=p+1}^n \text{Rm}^{N}(e_p,e_q,e_p,e_q)\Big|_{pr_N(x_\theta)} \\
    &+\sum\limits_{p=1}^m\sum\limits_{q=p+1}^n\frac{1}{\epsilon}\cos^2\theta \left(h_{\partial N_1} \KulNo (\nu^\flat\otimes \nu^\flat ) \right)(e_p,e_q,e_p,e_q)\Big|_{pr_N(x_\theta)}+O(1)
\end{align*}

The result then follows from Proposition \ref{Proposition:Convexity-LinearAlgebra} by choosing $\epsilon>0$ sufficient small.


\begin{thebibliography}{99}

\bibitem{Brendle-Hirsch-Johne}
S.~Brendle, S.~Hirsch and F.~Johne, \textit{A generalization of Geroch's conjecture}, to appear in Communications on Pure and Applied Mathematics, arXiv:2207.08617


\bibitem{Brendle-Marques-Neves}
S.~Brendle, F.~C.~Marques and A.~Neves,
\textit{Deformations of the hemisphere that increase scalar curvature}, 
Inventiones Mathematicae (2011), no. 185, 175--197

\bibitem{Chen:2022:ArbitraryEnds}
S.~Chen, \textit{A Generalization of the Geroch conjecture with Arbitrary Ends}, to appear in Mathematische Annalen,
arXiv:2212.10014


\bibitem{Chow}
T.-K.~A.~Chow, \textit{Positivity of curvature on manifolds with boundary},
Int. Math. Res. Not., 15 (2022), 11401--11426

\bibitem{Chu-Kwong-Lee:2022:Rigidity}
J.~Chu, K.-K.~Kwong and M.-C.~Lee,
\textit{Rigidity on non-negative intermediate curvature},
arXiv:2008.12240


\bibitem{GromovLawson:Spin}
M.~Gromov and H.-B.~Lawson, \textit{Spin and scalar curvature in the presence of a fundamental group. I}, Annals of Mathematics (1980), 209--230
 
\bibitem{GromovLawson:1984:DiracOperator}
M.~Gromov and H.-B.~Lawson, \textit{Positive scalar curvature and the Dirac operator on complete Riemannian manifolds}, Inst. Hautes Etudes Sci. Publ. Math. (1983), no.~58, (1984), 83--196

\bibitem{Labbi:2023:Preordering_PSC}
M.~L.~Labbi, \textit{On a pre-ordering of compact PSC manifolds
and PSC Riemannian metrics}, arXiv:2301.05270

\bibitem{Miao:2002:PMT-Corners}
P.~Miao, \textit{Positive Mass Theorem on Manifolds admitting Corners along a Hypersurface}, Adv.~Theor.~Math.~Phys., 6 (2002), 1163--1182
  
\bibitem{Raede:2021:BandComparison}
D.~Räde, \textit{Scalar and mean curvature comparison
via $\mu$-bubbles}, arXiv:2104.10120

\bibitem{Schlichting:2014:Thesis}
A.~Schlichting, \textit{Smoothing singularities of Riemannian metrics while preserving lower curvature bounds}, Ph.D. thesis, Otto von Guericke University Magdeburg, 2014

\bibitem{Schoen-Yau:1979:Structure_PositiveScalar}
R.~Schoen and S.-T.~Yau, \textit{On the structure
of manifolds with positive scalar curvature},
Manuscr. Math. 28 (1979), 159--183

\bibitem{Schoen-Yau:1979:PMT}
R.~Schoen, and S.-T.~Yau, \textit{On the proof of the positive mass conjecture in general relativity}, Commun.Math. Phys. \textbf{65}, (1979), 45–76

\bibitem{Shi-Tam:2002:BoundaryBehaviour}
Y.~Shi and L.-F.~Tam, \textit{Positive Mass Theorem and the Boundary Behaviors of Compact Manifolds with Nonnegative Scalar Curvature}, J.~Differential.~Geom. 62 (2002), 79--125

\bibitem{Xu:2023:DimensionConstraints}
K.~Xu, 
\textit{Dimension Constraints in Some Problems Involving Intermediate Curvature}, arXiv:2301.02730

 
\end{thebibliography}
\end{document}